\documentclass[12pt]{amsart}
\usepackage[english]{babel}
\usepackage{amsmath}
\usepackage{amsthm}
\usepackage{anysize}
\usepackage{mathtools}
\usepackage{mathrsfs}
\usepackage{amssymb}
\usepackage{xcolor}
\usepackage{xfrac}
\usepackage{esint}
\usepackage{graphicx}
\usepackage{float}
\usepackage{array}
\usepackage{enumitem}
\usepackage[new]{old-arrows}
\usepackage[all,cmtip]{xy}
\usepackage{tikz-cd}
\usepackage{centernot}
\usepackage{stmaryrd}
\usepackage{comment}

\newtheorem{theorem}{Theorem}[section]
\newtheorem{corollary}[theorem]{Corollary}
\newtheorem{lemma}[theorem]{Lemma}

\newtheorem{definition}[theorem]{Definition}

\DeclareMathOperator{\Aut}{Aut}

\DeclareMathOperator{\Hom}{Hom}
\DeclareMathOperator{\Vol}{Vol}

\begin{document}
	\title{Approximation Schemes for Path Integration on Riemannian Manifolds}
	\author{Juan Carlos Sampedro} \thanks{The author has been supported by the Research Grant PGC2018-097104-B-I00 of the Spanish Ministry of Science, Technology and Universities, by the Institute of Interdisciplinar Mathematics of Complutense University and by PhD Grant PRE2019\_1\_0220 of the Basque Country Government.}
	\address{Department of Mathematical Analysis and Applied Mathematics \\
		Faculty of Mathematical Science \\
		Complutense University of Madrid \\
		28040-Madrid \\
		Spain.}
	\email{juancsam@ucm.es}

\begin{abstract}
	\centerline{\emph{Truth is much too complicated to allow anything but approximations.}}
	\centerline{John von Neumann}
	\vspace{0.5cm}
	In this paper, we prove a finite dimensional approximation scheme for the Wiener measure on closed Riemannian manifolds, establishing a generalization for $L^{1}$-functionals, of the approach followed by Andersson and Driver on \cite{1}. We follow a new approach motived by categorical concept of colimit.
\end{abstract}


\maketitle

\section{Introduction}

 In 1920,  N. Wiener, based on Daniell's recent interpretation of integral \cite{5}-\cite{7}, defined in \cite{35} an integral for bounded and continuous functionals $F:\mathcal{C}_{\mathbf{x}_{0}}[0,1]\to\mathbb{R}$, where the notation $\mathcal{C}_{\mathbf{x}_{0}}[0,1]$ stands for the space of continuous functions $\mathbf{u}:[0,1]\to\mathbb{R}$ with base point $\mathbf{u}(0)=\mathbf{x}_{0}$. In later papers \cite{35}-\cite{40}, he connected this notion to that of Brownian motion and he defined the so-called Wiener process. In posterior works, he generalized this results defining a probability measure $\mu_{\mathbf{x}_{0}}
 $ on the measurable space $(\mathcal{C}_{\mathbf{x}_{0}}[0,1],\mathcal{B}_{\mathbf{x}_{0}})$, where $\mathcal{B}_{\mathbf{x}_{0}}$ stands for the Borel $\sigma$-algebra of $\mathcal{C}_{\mathbf{x}_{0}}[0,1]$ endowed with the uniform convergence topology. This measure is characterized by the following property: For each finite partition $\mathcal{T}=\{t_{1}<t_{2}<\cdots<t_{n}\}\subset(0,1]$ and each family $(B_{t})_{t\in\mathcal{T}}$ of Borel subsets of $\mathbb{R}$, the identity
\begin{equation}\label{E1}
\mu_{\mathbf{x}_{0}}\left(\pi^{-1}_{\mathcal{T}}(B_{t})_{t\in \mathcal{T}} \right)=\int_{B_{t_{1}}}\overset{(n)}{\cdots}\int_{B_{t_{n}}}\prod_{i=1}^{n}p_{t_{i}-t_{i-1}}(x_{i},x_{i-1})\prod_{i=1}^{n}dx_{i}, \quad t_{0}=0, \ x_{0}=\mathbf{x}_{0},
\end{equation}
\noindent holds, where $p_{t}(x,y)$ is the heat kernel of $\mathbb{R}$ and $\pi_{\mathcal{T}}: \mathcal{C}_{\mathbf{x}_{0}}[0,1] \to \mathbb{R}^{\mathcal{T}}$ is the projector defined by 
$$\pi_{\mathcal{T}}: \mathcal{C}_{\mathbf{x}_{0}}[0,1] \longrightarrow \mathbb{R}^{\mathcal{T}}, \quad \pi_{\mathcal{T}}(\mathbf{u}):=(\mathbf{u}(t))_{t\in \mathcal{T}}.$$
\par At first, it seems there is no easy way to compute the integral of an arbitrary measurable functional $F:\mathcal{C}_{\mathbf{x}_{0}}[0,1]\to\mathbb{R}$. Nevertheless, Wiener proved in \cite{35} an analogue of Jessen's formula \cite{16,J} for the measure $\mu_{\mathbf{x}_{0}}$. More explicitly, he proved that given a bounded and continuous functional on $\mathcal{C}_{\mathbf{x}_{0}}[0,1]$ and a partition $\mathcal{T}=\left\{\mathcal{T}_{n}\right\}_{n\in\mathbb{N}}$, $\mathcal{T}_{n}=\{t^{n}_{i}\}_{i=1}^{n}$ of $[0,1]$ with mesh zero, $$\lim_{n\to\infty}\max_{2\leq i \leq n}|t_{i}^{n}-t_{i-1}^{n}|=0,$$ the integral of $F$ can be computed though finite dimensional integrals via
\begin{equation*}
\int_{\mathcal{C}_{\mathbf{x}_{0}}[0,1]}F(\mathbf{u})\ d\mu_{\mathbf{x}_{0}}(\mathbf{u})=\lim_{n\to\infty}\int_{\mathbb{R}^{n}}F_{n}(x_{1},x_{2},...,x_{n})\prod_{i=1}^{n}p_{t^{n}_{i}-t^{n}_{i-1}}(x_{i},x_{i-1})\prod_{i=1}^{n}dx_{i}, 
\end{equation*}
\noindent where the functions $F_{n}:\mathbb{R}^{n}\to\mathbb{R}$ are defined by $$F_{n}(x_{1},x_{2},...,x_{n}):=F(\mathbf{u}_{(x_{1},x_{2},...)}),$$ where $\mathbf{u}_{(x_{1},x_{2},...)}$ denotes the linear interpolation of the points $x_{1}, x_{2}, ..., x_{n}$, for each $n\in\mathbb{N}$. In \cite{J}, the author generalized this formula to every $L^{1}$-functional proving that for each $F\in L^{1}(\mathcal{C}_{\mathbf{x}_{0}}[0,1],\mu_{\mathbf{x}_{0}})$, there exists a finite dimensional functional sequence $(F_{n})_{n\in\mathbb{N}}\in\bigtimes_{n\in\mathbb{N}}L^{1}(\mathbb{R}^{n},\mu^{n}_{\mathbf{x}_{0}})$ such that
\begin{equation}\label{EE}
\int_{\mathcal{C}_{\mathbf{x}_{0}}[0,1]}F \ d\mu_{\mathbf{x}_{0}}=\lim_{n\to\infty}\int_{\mathbb{R}^{n}}F_{n} \ d\mu^{n}_{\mathbf{x}_{0}} \ \  \text{ where } \ d\mu^{n}_{\mathbf{x}_{0}}=\prod_{i=1}^{n}p_{t^{n}_{i}-t^{n}_{i-1}}(x_{i},x_{i-1})\prod_{i=1}^{n}dx_{i}.
\end{equation}

\par A similar discussion can be done for the category of Riemannian Manifolds. Given a compact connected Riemannian manifold $(M,g)$ (closed Riemannian manifold for short) of dimension $m$, we can construct, analogously to the case of $[0,1]$, the measure space $(\mathcal{C}_{\mathbf{x}_{0}}(M),\mu_{\mathbf{x}_{0}})$. Here, the notation $\mathcal{C}_{\mathbf{x}_{0}}(M)$ denotes the space of continuous curves $\gamma:[0,1]\to M$ beginning at $\mathbf{x}_{0}$ and $\mu_{\mathbf{x}_{0}}$ the Wiener measure on $\mathcal{C}_{\mathbf{x}_{0}}(M)$, i.e., a measure satisfying an analogue of equation \eqref{E1} for this setting (see section 2 for further details). Similar versions of Jessen type formula have been developed for the category of Riemannian manifolds in \cite{1}. In that article, Andersson and Driver proved that given a bounded and continuous functional $F:\mathcal{C}_{\mathbf{x}_{0}}(M)\to\mathbb{R}$, the identity
\begin{equation}\label{E2}
\int_{\mathcal{C}_{\mathbf{x}_{0}}(M)}F \ d\mu_{\mathbf{x}_{0}}=\lim_{n\to \infty}\int_{H_{\mathcal{T}_{n}}(M)}F(\sigma) \ d\nu_{\mathcal{T}_{n}}(\sigma)
\end{equation}
holds, where $(H_{\mathcal{T}_{n}}(M),\nu_{\mathcal{T}_{n}})$ is a finite dimensional measure space based on the geometrical data of $(M,g)$ and $\mathcal{T}=\left\{\mathcal{T}_{n}\right\}_{n\in\mathbb{N}}, \mathcal{T}_{n}= \{t^{n}_{i}\}_{i=1}^{n}$, is a partition of $[0,1]$ with mesh zero. Roughly, $H_{\mathcal{T}_{n}}(M)$ is the space of piecewise geodesics paths in $M$, $\sigma:[0,1]\to M$, which change directions only at the partition points $\mathcal{T}_{n}= \{t^{n}_{i}\}_{i=1}^{n}$. The precise definition of the pair $(H_{\mathcal{T}_{n}}(M),\nu_{\mathcal{T}_{n}})$ will be given in section 6.
\par After Andersson and Driver, several developments and generalizations have been done by several authors, for instance this scheme has been generalized to heat kernels on vector bundles in \cite{B1,B2} and further developed in \cite{Li,L1,L2,L3}. This finite dimensional approximations of path integrals have a great impact in theoretical physics and in particular are extremely useful for the Feynman path integral approach to quantum field theory \cite{10,11}.
\par The aim of this article is to establish a generalization of equation \eqref{E2} for every integrable functional $F:\mathcal{C}_{\mathbf{x}_{0}}(M)\to\mathbb{R}$, not necessarily bounded and continuous in the vein of the analogous result \eqref{EE} for the classical Wiener measure proved in \cite{J}. To obtain this result, we use a categorical point of view of this type of approximations. We prove that for $1\leq p<\infty$, the space $L^{p}(\mathcal{C}_{\mathbf{x}_{0}}(M),\mu_{\mathbf{x}_{0}})$ is the colimit of certain diagram consisting on $L^{p}$-spaces of finite dimensional data associated with $M$.
\par This generalization and structural result has many applications in stochastic analysis. For instance, this scheme allows to embed the notion of the Stratonovich stochastic integral $$\int_{0}^{1} f(X_{t})\circ dX_{t}\in L^{2}(\mathcal{C}_{\mathbf{x}_{0}}(M),\mu_{\mathbf{x}_{0}}),$$ in our approximation techniques, where $X$ denotes the Wiener process on $(\mathcal{C}_{\mathbf{x}_{0}}(M),\mu_{\mathbf{x}_{0}})$. Indeed, since the Stratonovich stochastic integral is not, in general, a continuous and bounded functional, we cannot apply the approximation \eqref{E2} directly, but we can apply the techniques of this article. 
\par The paper is organized as follows. In section two, we present the abstract results about colimits in category theory that will be used throughout this article and we recall the definition and construction of the Wiener measure on closed Riemannan manifolds. In section three, we prove that for $1\leq p<\infty$, the space $L^{p}(\mathcal{C}_{\mathbf{x}_{0}}(M),\mu_{\mathbf{x}_{0}})$ is the colimit of a diagram consisting on $L^{p}$-spaces of Cartesian products of a finite number of $M$-factors. In section four, we derive a finite dimensional approximation formula of type \eqref{EE} for integrable functionals $F\in L^{1}(\mathcal{C}_{\mathbf{x}_{0}}(M),\mu_{\mathbf{x}_{0}})$, via a particular realization of the colimit of the diagram defined in the preceding section. In section five, we apply our approximation scheme to Stratonovich stochastic integrals. The end of this article, section six, is dedicated to recall Andersson and Driver's scheme and to adapt the results developed in sections three and four to their framework. In particular, we obtain a generalization of \eqref{E2} for every integrable functional.
\par It is convenient to remark that all the results obtained in this article can be easily adapted to the category of Riemannian manifolds with boundary and to the case of continuous paths with fixed initial and end point, the so-called Pinned Wiener spaces. See for instance \cite{B3} for the definition and construction of the Wiener measure on this spaces.

\section{Preliminaries}

In this short section, we will recall some basic facts about colimits and the Wiener measure on Riemannian manifolds that will be used trough this article. It will be also useful to fix notation.

\subsection{Basic notions about colimits}

Let $\mathfrak{C}$ be a fixed category and $I$ a directed set. A diagram is a functor $F:I\to \mathfrak{C}$. It can be represented as $(X_{i},\varphi_{ij})$ for a family of objects indexed by $I$, $\{X_{i}:i\in I\}$, and for each $i\leq j$ a morphism $\varphi_{ij}:X_{i}\to X_{j}$ such that $\varphi_{ii}$ is the identity on $X_{i}$ and $\varphi_{ik}=\varphi_{jk}\circ \varphi_{ij}$ for each $i\leq j\leq k$. In general, we can consider diagrams as functors $F:\mathfrak{J}\to\mathfrak{C}$ indexed by a general category $\mathfrak{J}$, however, diagrams indexed by directed sets are enough for our purposes.
A cocone is a pair $(X,\phi_{i})$ where $X$ is an object of $\mathfrak{C}$ and $\phi_{i}:X_{i}\to X$ is a morphism of $\mathfrak{C}$ such that $\phi_{i}=\phi_{j}\circ \varphi_{ij}$ for every $i\leq j$. The class of cocones of a given diagram forms itself a category. A colimit of the diagram $(X_{i},\varphi_{ij})$ is defined to be a cocone $(L, \psi_{i})$ characterized by the following universal property: For any cocone $(X,\phi_{i})$, there exists a unique morphism $\varphi_{X}:L\to X$ making the diagram of Figure 1 commutative.
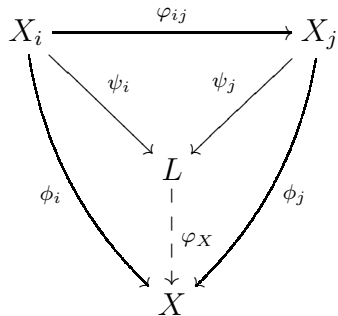
\begin{figure}[h!]
\[
\xymatrix@C+1em@R+1em{ 
	X_i \ar^{\varphi_{ij}}[rr] \ar^{\psi_i}[dr] \ar@/_1em/_{\phi_i}[ddr] & & X_j \ar_{\psi_j}[dl] \ar@/^1em/^{\phi_j}[ddl] \\
	&L \ar@{-->}^{\varphi_{X}}[d] & \\
	& X &
}
\]
\caption{Diagram of the Universal Property}
\label{F1}
\end{figure}

The colimits of a given diagram can be also characterized as initial objects in the category of cocones. The colimit of a given diagram does not necessarily exist, however if the colimit exists, it is unique up to a unique isomorphism in the category of cocones. Any given isomorphism class of the colimit is called a \textit{realization}. The colimit of a given diagram $(X_{i},\varphi_{ij})$ is commonly denoted as $L\equiv \underrightarrow{\lim} \ X_{i}$.
\par In this article, we will consider the category $\mathfrak{Ban}$ whose objects are Banach spaces and whose morphisms are linear isometries. The proof of the existence of the colimit for every diagram on $\mathfrak{Ban}$ can be found in \cite[App. L]{WO} or in \cite{JC} and references therein.
\par Though this article, we will make use of the following simple Lemma.
\begin{lemma}
	\label{L2.1}
	Let $(X_{i},\varphi_{ij})$ be a diagram on $\mathfrak{Ban}$. Then a cocone $(X,\phi_{i})$ is a realization of the colimit of $(X_{i},\varphi_{ij})$ if and only if $\bigcup_{i\in I}\phi_{i}(X_{i})$ is dense in $X$.
\end{lemma}

\begin{proof}
	Suppose that $(X,\phi_{i})$ defines a realization of the colimit of $(X_{i},\varphi_{ij})$. Consider the cocone $(Y,\rho_{i})$ where
	\begin{equation*}
	Y:=\overline{\bigcup_{i\in I}\phi_{i}(X_{i})}\subset X,
	\end{equation*}
	and where $\rho_{i}:X_{i}\to Y$ are the canonical inclusions. By the universal property there exists a unique isometry $\varphi_{Y}:X\to Y$ making the diagram of Figure \ref{F1} commutative. It is straightforward to prove that the morphism $\varphi_{Y}$ is in fact an isomorphism, hence $\bigcup_{i\in I}\phi_{i}(X_{i})$ is dense in $X$. On the other hand, take a cocone $(X,\phi_{i})$ with $\bigcup_{i\in I}\phi_{i}(X_{i})$ dense in $X$. For any other cocone $(Y,\rho_{i})$, define the morphism $\varphi_{Y}:X\to Y$ by 
	\begin{equation*}
	\varphi_{Y}\circ\phi_{i}:=\rho_{i}, \quad i\in I.
	\end{equation*}
	Then it is easy to see that the morphism $\varphi_{Y}$ is the unique morphism making the diagram of Figure \ref{F1} commutative. Hence $(X,\phi_{i})$ is a colimit. This concludes the proof.
\end{proof}

\subsection{Wiener measure on Riemannian manifolds}
 
Along this article, $(M,g)$ denotes a compact connected Riemannian manifold (closed Riemannian manifold for short) with a fixed base point $\mathbf{x}_{0}\in M$. The notation $\mathcal{C}_{\mathbf{x}_{0}}(M)$ stands for the space of continuous paths $\gamma\in\mathcal{C}([0,1],M)$ satisfying $\gamma(0)=\mathbf{x}_{0}$, $\mathcal{B}_{\mathbf{x}_{0}}$ for the Borel $\sigma$-algebra of $\mathcal{C}_{\mathbf{x}_{0}}(M)$ with respect to the uniform convergence topology given by the induced metric of $M$ and $\mu_{\mathbf{x}_{0}}$ for the Wiener measure on $M$ with base point $\mathbf{x}_{0}$.
\par We recall the definition of the measure space $(\mathcal{C}_{\mathbf{x}_{0}}(M),\mathcal{B}_{\mathbf{x}_{0}},\mu_{\mathbf{x}_{0}})$. Consider in $(M,g)$ the measure $\mu: \mathcal{B}_{M}\to [0,+\infty]$ induced by the metric $g$, where $\mathcal{B}_{M}$ denotes the Borel $\sigma$-algebra of $M$. This measure is locally given by the expression
\begin{equation*}
d\mu=\sqrt{\det(g_{ij})_{ij}} \ dx_{1}\wedge \cdots \wedge dx_{m}
\end{equation*}
where $m$ is the dimension of $M$ and $(g_{ij})_{ij}$ is the matrix of $g$ in a local chart. For each closed Riemannian manifold $(M,g)$, there exists a heat kernel $p_{t}(x,y)$, for $t>0$, $x,y\in M$, i.e., the Schwartz kernel of the selfadjoint operator $e^{t\Delta}$ on $L^{2}(M,\mu)$, where $\Delta$ denotes the Laplace-Beltrami operator on $(M,g)$. The proof of the existence of this map can be found in \cite{B3, 1.7}. A well-known consequence of the Kolmogorov extension Theorem \cite[Th. 6.1]{Y}, is the existence of a probability measure $$\mu_{\mathbf{x}_{0}}:\mathcal{B}\longrightarrow[0,+\infty]$$ on $(M^{[0,1]},\mathcal{B})$, where $\mathcal{B}$ denotes the Borel $\sigma$-algebra of $M^{[0,1]}$ with respect to the product topology, satisfying the identity
\begin{equation}\label{E9}
\mu_{\mathbf{x}_{0}}\left(\pi^{-1}_{\mathcal{T}}(B_{t})_{t\in \mathcal{T}\backslash\{0\}} \right)=\int_{B_{t_{1}}}\overset{(n)}{\cdots}\int_{B_{t_{n}}}\prod_{j=1}^{n}p_{t_{j}-t_{j-1}}(x_{j},x_{j-1})\prod_{j=1}^{n}d\mu(x_{j}), \quad x_{0}=\mathbf{x}_{0},
\end{equation}
for each finite partition $\mathcal{T}=\{0=t_{0}<t_{1}<\cdots<t_{n}\}\subset[0,1]$ and each family of Borel subsets $(B_{t})_{t\in\mathcal{T}\backslash\{0\}}\subset \mathcal{B}_{M}$. Here and in the sequel, the notation $\pi_{\mathcal{T}}$ stands for the projector defined by
\begin{equation}\label{QQ}
\pi_{\mathcal{T}}:M^{[0,1]} \longrightarrow M^{\mathcal{T}\backslash\{0\}}, \quad \pi_{\mathcal{T}}(\gamma_{t})_{t\in[0,1]}:=(\gamma_{t})_{t\in\mathcal{T}\backslash\{0\}}.
\end{equation}
Since $(M,g)$ is compact, it is, in particular, stochastically complete (see for instance \cite{1.7}), and therefore
$$\int_{M}p_{t}(x,y)d\mu(y)=1$$ 
for each $t>0$ and $x\in M$. This fact implies that the measure $\mu_{\mathbf{x}_{0}}$ is of probability, that is, $\mu_{\mathbf{x}_{0}}(M^{[0,1]})=1$. 
\par On the other hand, by \cite[Cr. 2.19]{B3}, the measure $\mu_{\mathbf{x}_{0}}$ satisfies the identity $\mu_{\mathbf{x}_{0}}(\mathcal{C}^{\alpha}_{\mathbf{x}_{0}}(M))=1 \text{ for each } \alpha\in(0,1/2)$ where $\mathcal{C}^{\alpha}_{\mathbf{x}_{0}}(M)$ stands for the subset of $M^{[0,1]}$ consisting of H\"older continuous paths $\gamma:[0,1]\to M$ of exponent $\alpha\in(0,1)$ satisfying $\gamma(0)=\mathbf{x}_{0}$. Therefore, since $\mathcal{C}_{\mathbf{x}_{0}}(M)$ is a Borel subset of $M^{[0,1]}$ (see \cite[Th. 10.28]{F}) containing $\mathcal{C}^{\alpha}_{\mathbf{x}_{0}}(M)$ and since
\begin{equation*}
\mathcal{B} \cap \mathcal{C}_{\mathbf{x}_{0}}(M):=\{B\cap \mathcal{C}_{\mathbf{x}_{0}}(M):B\in \mathcal{B}\},
\end{equation*}
coincides with the Borel $\sigma$-algebra $\mathcal{B}_{\mathbf{x}_{0}}$ (see \cite[Prop. 2.2]{I}), we can consider the restricted probability space $(\mathcal{C}_{\mathbf{x}_{0}}(M),\mathcal{B}_{\mathbf{x}_{0}},\mu_{\mathbf{x}_{0}})$. The restricted measure $\mu_{\mathbf{x}_{0}}$ is called the Wiener measure of $M$ with base point $\mathbf{x}_{0}$. The proof of these facts can be found for instance in \cite{B3,1.7} and references therein.

\section{Approximation Scheme for the Wiener Measure} 

In this section, we will prove that the spaces $L^{p}(\mathcal{C}_{\mathbf{x}_{0}}(M),\mu_{\mathbf{x}_{0}})$, $1\leq p<\infty$, are realizations of the colimit of the diagram $(L^{p}(M^{\mathcal{T}},\mu_{\mathbf{x}_{0}}^{\mathcal{T}}),\pi^{\ast}_{\mathcal{T}\mathcal{T}'})$. Let us start by defining the diagram $(L^{p}(M^{\mathcal{T}},\mu_{\mathbf{x}_{0}}^{\mathcal{T}}),\pi^{\ast}_{\mathcal{T}\mathcal{T}'})$. Consider the directed set $\mathcal{P}$ consisting on partitions of $[0,1]$ $$\mathcal{T}=\{0=t_{0}<t_{1}<\cdots<t_{n}\},$$ partially ordered by inclusion. For notational simplicity, we will use the notation $M^{\mathcal{T}}$ to denote 
\begin{equation*}
M^{\mathcal{T}}:=\bigtimes_{t\in\mathcal{T}\backslash\{0\}}M.
\end{equation*}
For a partition $\mathcal{T}=\{0=t_{0}<t_{1}<\cdots<t_{n}\}$ in $\mathcal{P}$, consider the probability space $(M^{\mathcal{T}},\bigotimes_{i=1}^{n}\mathcal{B}_{M},\mu_{\mathbf{x}_{0}}^{\mathcal{T}})$ where $\mu_{\mathbf{x}_{0}}^{\mathcal{T}}$ is the measure defined by the density
\begin{equation*}
d\mu_{\mathbf{x}_{0}}^{\mathcal{T}}=\prod_{i=1}^{n}p_{t_{i}-t_{i-1}}(x_{t_{i}},x_{t_{i-1}})\prod_{i=1}^{n}d\mu(x_{t_{i}}), \quad x_{0}=\mathbf{x}_{0},
\end{equation*}
and $\bigotimes_{i=1}^{n}\mathcal{B}_{M}$ denotes the product $\sigma$-algebra. For each pair of partitions $\mathcal{T},\mathcal{T}'$ with $\mathcal{T}\subset \mathcal{T}'$, we obtain a measurable projection map 
\begin{equation*}
\pi_{\mathcal{T}\mathcal{T}'}:M^{\mathcal{T}'}\longrightarrow M^{\mathcal{T}}, \quad \pi_{\mathcal{T}\mathcal{T}'}(x_{t})_{t\in\mathcal{T}'\backslash\{0\}}:=(x_{t})_{t\in\mathcal{T}\backslash\{0\}}.
\end{equation*}
We therefore obtain a diagram in the category of measure spaces, indexed by the directed set $\mathcal{P}$. For each pair of partitions $\mathcal{T},\mathcal{T}'$ with $\mathcal{T}\subset \mathcal{T}'$, we can construct the isometry
\begin{equation*}
\pi^{\ast}_{\mathcal{T}\mathcal{T}'}: L^{p}(M^{\mathcal{T}},\mu^{\mathcal{T}}_{\mathbf{x}_{0}})\longrightarrow L^{p}(M^{\mathcal{T}'},\mu^{\mathcal{T}'}_{\mathbf{x}_{0}}), \quad \pi^{\ast}_{\mathcal{T}\mathcal{T}'}(f):=f\circ\pi_{\mathcal{T}\mathcal{T}'}.
\end{equation*}
Applying this construction to the entire diagram, we obtain a diagram in $\mathfrak{Ban}$ indexed by the directed set $\mathcal{P}$, where all arrow directions are now swapped, as the $L^{p}$-construction is contravariant. This diagram forms a diagram in $\mathfrak{Ban}$ that will be subsequently denoted by $(L^{p}(M^{\mathcal{T}},\mu^{\mathcal{T}}_{\mathbf{x}_{0}}),\pi^{\ast}_{\mathcal{T}\mathcal{T}'})$. 
\par It will be proved that the cocone $(L^{p}(\mathcal{C}_{\mathbf{x}_{0}}(M),\mu_{\mathbf{x}_{0}}),\phi_{\mathcal{T}})$, where $\phi_{\mathcal{T}}$ are the morphisms defined by
\begin{equation*}
\phi_{\mathcal{T}}:L^{p}(M^{\mathcal{T}},\nu^{\mathcal{T}}_{\mathbf{x}_{0}})\longrightarrow L^{p}(\mathcal{C}_{\mathbf{x}_{0}}(M),\mu_{\mathbf{x}_{0}}), \quad \phi_{\mathcal{T}}(f):=f\circ\pi_{\mathcal{T}}|_{\mathcal{C}_{\mathbf{x}_{0}}(M)},
\end{equation*}
defines a realization of the colimit of $(L^{p}(M^{\mathcal{T}},\mu^{\mathcal{T}}_{\mathbf{x}_{0}}),\pi^{\ast}_{\mathcal{T}\mathcal{T}'})$.
The following Lemma will be useful for this purpose.

\begin{lemma}\label{L4.1}
Let $1\leq p < \infty$, then the subspace $\bigcup_{\mathcal{T}\in\mathcal{P}}\phi_{\mathcal{T}}(L^{p}(M^{\mathcal{T}},\mu_{\mathbf{x}_{0}}^{\mathcal{T}}))$ is dense in $L^{p}(\mathcal{C}_{\mathbf{x}_{0}}(M),\mu_{\mathbf{x}_{0}})$.
\end{lemma}

\begin{proof}
	 As we have recalled in section 2.2, the following equality of $\sigma$-algebras
	 \begin{equation}\label{E17}
	 \mathcal{B}_{\mathbf{x}_{0}}=\mathcal{B}\cap\mathcal{C}_{\mathbf{x}_{0}}(M),
	 \end{equation}
	 holds. By definition, the product $\sigma$-algebra $\mathcal{B}$ is generated by the cylinder sets and therefore by \eqref{E17}, it follows that
	 \begin{equation*}
	 \mathcal{B}_{\mathbf{x}_{0}}=\sigma(\mathcal{R}), \quad \mathcal{R}:=\{\pi^{-1}_{\mathcal{T}}(B_{t})_{t\in\mathcal{T}}:(B_{t})_{t\in\mathcal{T}}\subset\mathcal{B}_{M}, \mathcal{T}\subset (0,1] \text{ finite}\}\cap \mathcal{C}_{\mathbf{x}_{0}}(M).
	 \end{equation*}
	 Since $(\mathcal{C}_{\mathbf{x}_{0}}(M),\mu_{\mathbf{x}_{0}})$ has finite measure, by \cite[Lem. 3.4.6]{C}, the subspace $\text{Span}\left\{\chi_{R}:R\in\mathcal{R}\right\}$, where $\chi_{R}$ denoted the characteristic function of $R$, is dense in $L^{p}\left(\mathcal{C}_{\mathbf{x}_{0}}(M),\mu_{\mathbf{x}_{0}}\right)$. From the inclusion
	\begin{equation*}
	\left\{\chi_{R}:R\in\mathcal{R}\right\}\subset\bigcup_{\mathcal{T}\in\mathcal{P}}\phi_{\mathcal{T}}(L^{p}(M^{\mathcal{T}},\mu_{\mathbf{x}_{0}}^{\mathcal{T}})),
	\end{equation*}
    it follows that $\bigcup_{\mathcal{T}\in\mathcal{P}}\phi_{\mathcal{T}}(L^{p}(M^{\mathcal{T}},\mu_{\mathbf{x}_{0}}^{\mathcal{T}}))$ is dense in $L^{p}\left(\mathcal{C}_{\mathbf{x}_{0}}(M),\mu_{\mathbf{x}_{0}}\right)$. This
	concludes the proof.
\end{proof}

Finally, we prove the main result of this section.

\begin{theorem}\label{th4.7}
	Let $1\leq p < \infty$, then the cocone $(L^{p}(\mathcal{C}_{\mathbf{x}_{0}}(M),\mu_{\mathbf{x}_{0}}),\phi_{\mathcal{T}})$ defines a realization of the colimit of $(L^{p}(M^{\mathcal{T}},\mu^{\mathcal{T}}_{\mathbf{x}_{0}}),\pi^{\ast}_{\mathcal{T}\mathcal{T}'})$.
\end{theorem}

\begin{proof}
	By Lemma \ref{L4.1}, the subspace 
	\begin{equation*}
	\bigcup_{\mathcal{T}\in\mathcal{P}}\phi_{\mathcal{T}}(L^{p}(M^{\mathcal{T}},\mu_{\mathbf{x}_{0}}^{\mathcal{T}}))\subset L^{p}(\mathcal{C}_{\mathbf{x}_{0}}(M),\mu_{\mathbf{x}_{0}})
	\end{equation*}
	is dense in $L^{p}(\mathcal{C}_{\mathbf{x}_{0}}(M),\mu_{\mathbf{x}_{0}})$. Therefore, we can apply Lemma \ref{L2.1} to $(L^{p}(\mathcal{C}_{\mathbf{x}_{0}}(M),\mu_{\mathbf{x}_{0}}),\phi_{\mathcal{T}})$ obtaining the result.
\end{proof}

\section{Derivation of the Limit Formula}


In this section, we will prove that the integral of a given integrable functional can be expressed as the limit of finite dimensional integrals. To relate the colimit structure obtained in the last section with the integration procedure, we have to provide a particular realization of the colimit of the directed system $(L^{p}(M^{\mathcal{T}},\mu^{\mathcal{T}}_{\mathbf{x}_{0}}),\pi^{\ast}_{\mathcal{T}\mathcal{T}'})$. To define this realization, we need the following definition.
\par An element $(f_{\mathcal{T}})_{\mathcal{T}\in\mathcal{P}}\in\bigtimes_{\mathcal{T}\in\mathcal{P}}L^{p}(M^{\mathcal{T}},\mu^{\mathcal{T}}_{\mathbf{x}_{0}})$ is said to be co-Cauchy if for each $\varepsilon>0$, there exists $\mathcal{R}\in\mathcal{P}$ such that 
$$\|\pi^{\ast}_{\mathcal{T}\mathcal{T}'}(f_{\mathcal{T}})-f_{\mathcal{T'}}\|_{L^{p}_{\mathcal{T}}}<\varepsilon, \text{ for all } \mathcal{T},\mathcal{T}'\in\mathcal{P} \text{ with } \mathcal{R}\subset\mathcal{T}\subset\mathcal{T}'.$$
\noindent For notational simplicity we denote $L^{p}\equiv L^{p}(\mathcal{C}_{\mathbf{x}_{0}}(M),\mu_{\mathbf{x}_{0}})$ and $L^{p}_{\mathcal{T}}\equiv L^{p}(M^{\mathcal{T}},\mu_{\mathbf{x}_{0}}^{\mathcal{T}})$.

\noindent We define the space $\mathfrak{L}( L^{p}(M^{\mathcal{T}},\mu^{\mathcal{T}}_{\mathbf{x}_{0}}))$ by
\begin{equation}
\label{E20}
\mathfrak{L}( L^{p}(M^{\mathcal{T}},\mu^{\mathcal{T}}_{\mathbf{x}_{0}})):=\Big\{(f_{\mathcal{T}})_{\mathcal{T}\in\mathcal{P}}\in\bigtimes_{\mathcal{T}\in\mathcal{P}}L^{p}(M^{\mathcal{T}},\mu^{\mathcal{T}}_{\mathbf{x}_{0}}) : (f_{\mathcal{T}})_{\mathcal{T}\in\mathcal{P}} \text{ is co-Cauchy}\Big\}\Big / \sim,
\end{equation}
where we relate $(f_{\mathcal{T}})_{\mathcal{T}\in\mathcal{P}}\sim (g_{\mathcal{T}})_{\mathcal{T}\in\mathcal{P}}$ if $$\lim_{\mathcal{T}}\|f_{\mathcal{T}}-g_{\mathcal{T}}\|_{L^{p}_{\mathcal{T}}}=0.$$ We define in $\mathfrak{L}(L^{p}(M^{\mathcal{T}},\mu^{\mathcal{T}}_{\mathbf{x}_{0}}))$ the norm $$\|(f_{\mathcal{T}})_{\mathcal{T}\in\mathcal{P}}\|_{\mathfrak{L}}:=\lim_{\mathcal{T}}\|f_{\mathcal{T}}\|_{L^{p}_{\mathcal{T}}}.$$ 
All the limits involved are considered as limits of nets. For further considerations, it is imperative to obtain the completeness of the space $\mathfrak{L}( L^{p}(M^{\mathcal{T}},\mu^{\mathcal{T}}_{\mathbf{x}_{0}}))$, which is not obvious from the definition. For this reason, we identify it directly with $L^{p}(\mathcal{C}_{\mathbf{x}_{0}}(M),\mu_{\mathbf{x}_{0}})$ via the following result.

\begin{theorem}
	\label{T1}
	The following spaces are isometrically isomorphic for $1\leq p<\infty$
	\begin{equation*}
	\mathfrak{L}( L^{p}(M^{\mathcal{T}},\mu^{\mathcal{T}}_{\mathbf{x}_{0}}))\simeq L^{p}(\mathcal{C}_{\mathbf{x}_{0}}(M),\mu_{\mathbf{x}_{0}}).
	\end{equation*}
\end{theorem}
\begin{proof}
	Let us prove that the map
	\begin{equation*}
	\mathfrak{I}_{p}:\mathfrak{L}( L^{p}(M^{\mathcal{T}},\mu^{\mathcal{T}}_{\mathbf{x}_{0}}))\longrightarrow L^{p}(\mathcal{C}_{\mathbf{x}_{0}}(M),\mu_{\mathbf{x}_{0}}), \quad (f_{\mathcal{T}})_{\mathcal{T}\in\mathcal{P}}\mapsto \lim_{\mathcal{T}}\phi_{\mathcal{T}}(f_{\mathcal{T}}).
	\end{equation*}
	defines an isometric isomorphism. The map $\mathfrak{I}_{p}$ is well defined since the net $(\phi_{\mathcal{T}}(f_{\mathcal{T}}))_{\mathcal{T}\in\mathcal{P}}$ converges in $L^{p}(\mathcal{C}_{\mathbf{x}_{0}}(M),\mu_{\mathbf{x}_{0}})$, a fact that follows from the identity
	\begin{align*}
	\|\phi_{\mathcal{T}}(f_{\mathcal{T}})-\phi_{\mathcal{T}'}(f_{\mathcal{T}'})\|_{L^{p}}&=\|(\phi_{\mathcal{T}'}\circ \pi^{\ast}_{\mathcal{T}\mathcal{T}'})(f_{\mathcal{T}})-\phi_{\mathcal{T}'}(f_{\mathcal{T}'})\|_{L^{p}}\\
	&=\|\pi^{\ast}_{\mathcal{T}\mathcal{T}'}(f_{\mathcal{T}})-f_{\mathcal{T}'}\|_{L^{p}_{\mathcal{T}'}}, \quad \mathcal{T}\subset \mathcal{T}'
	\end{align*}
	and the co-Cauchy property. It is easily seen that $\mathfrak{I}_{p}$ is an isometry since
	\begin{equation*}
	\|\mathfrak{I}_{p}(f_{\mathcal{T}})_{\mathcal{T}\in\mathcal{P}}\|_{L^{p}}=\|\lim_{\mathcal{T}} \phi_{\mathcal{T}}(f_{\mathcal{T}})\|_{L^{p}}=\lim_{\mathcal{T}}\|f_{\mathcal{T}}\|_{L^{p}_{\mathcal{T}}}=\|(f_{\mathcal{T}})_{\mathcal{T}\in\mathcal{P}}\|_{\mathfrak{L}}.
	\end{equation*}
	Finally, we prove that $\mathfrak{I}_{p}$ is onto. Let $f\in L^{p}(\mathcal{C}_{\mathbf{x}_{0}}(M),\mu_{\mathbf{x}_{0}})$ and take $\mathcal{D}=\{t_{i}\}_{i\in\mathbb{N}\cup\{0\}}$, $t_{0}=0$, a dense countable subset of $[0,1]$. Define the  sequence of partitions $\mathcal{P}=\{\mathcal{P}_{n}\}_{n\in\mathbb{N}}$ by 
	\begin{equation*}
	\mathcal{P}_{n}:=\{t_{i}\}_{i=0}^{n}, \quad n\in\mathbb{N}.
	\end{equation*}
	Clearly $\mathcal{P}_{n}\subset \mathcal{P}_{n+1}$, $n\in\mathbb{N}$ and $\bigcup_{n\in\mathbb{N}}\mathcal{P}_{n}=\mathcal{D}$. It is easily seen using the same techniques of the proof of Lemma \ref{L4.1}, that $\bigcup_{n\in\mathbb{N}}\phi_{\mathcal{P}_{n}}(L^{p}(M^{\mathcal{P}_{n}},\mu_{\mathbf{x}_{0}}^{\mathcal{P}_{n}}))$ is dense in $L^{p}(\mathcal{C}_{\mathbf{x}_{0}}(M),\mu_{\mathbf{x}_{0}})$. Hence there exists a sequence $$(f_{n})_{n\in\mathbb{N}}, \quad f_{n}\in L^{p}(M^{\mathcal{P}_{N_{n}}},\mu_{\mathbf{x}_{0}}^{\mathcal{P}_{N_{n}}}) \text{ for some } N_{n}\in\mathbb{N},$$ such that 
	\begin{equation*}
	\lim_{n\to \infty} \phi_{\mathcal{P}_{N_{n}}}(f_{n})=f \text{ in } L^{p}(\mathcal{C}_{\mathbf{x}_{0}}(M),\mu_{\mathbf{x}_{0}}).
	\end{equation*}
	Take an strictly increasing sequence $(M_{n})_{n\in\mathbb{N}}\subset \mathbb{N}$ such that $N_{n}\leq M_{n}, n\in\mathbb{N}$. Then, in particular, $\mathcal{P}_{N_{n}}\subset\mathcal{P}_{M_{n}}$ and $$g_{n}:=\pi^{\ast}_{\mathcal{P}_{N_{n}}\mathcal{P}_{M_{n}}}(f_{n})\in L^{p}(M^{\mathcal{P}_{M_{n}}},\mu_{\mathbf{x}_{0}}^{\mathcal{P}_{M_{n}}}), \quad n\in\mathbb{N}.$$ Define the sequence 
	\begin{equation*}
	(F_{n})_{n\in\mathbb{N}}, \quad F_{n}:=\left\{
	\begin{array}{ll}
	\pi^{\ast}_{\mathcal{P}_{M_{m}}\mathcal{P}_{n}}(g_{n}) & \text{ if } M_{m}\leq n <M_{m+1} \\
	0 & \text{ if } n< M_{1}
	\end{array}
	\right.
	\end{equation*}
	then 
	\begin{equation*}
	(F_{n})_{n\in\mathbb{N}}\in\bigtimes_{n\in\mathbb{N}}L^{p}(M^{\mathcal{P}_{n}},\mu_{\mathbf{x}_{0}}^{\mathcal{P}_{n}}) \text{ and } \lim_{n\to \infty} \phi_{\mathcal{P}_{n}}(F_{n})=f \text{ in } L^{p}(\mathcal{C}_{\mathbf{x}_{0}}(M),\mu_{\mathbf{x}_{0}}).
	\end{equation*}
	Define the associated net $(F_{\mathcal{T}})_{\mathcal{T}\in\mathcal{P}}$ by
	\begin{equation*}
	F_{\mathcal{T}}:=\left\{
	\begin{array}{ll}
	\pi^{\ast}_{\mathcal{P}_{n}\mathcal{T}}(F_{n}) & \text{ if } \mathcal{P}_{n}\subset \mathcal{T} \text{ but } \mathcal{P}_{n+1}\not\subset \mathcal{T} \\
	0 & \text{ else }
	\end{array}
	\right.
	\end{equation*}
	Finally, we have that $(F_{\mathcal{T}})_{\mathcal{T}\in\mathcal{P}}\in \bigtimes_{\mathcal{T}\in\mathcal{P}}L^{p}(M^{\mathcal{T}},\mu_{\mathbf{x}_{0}}^{\mathcal{T}})$ and $\lim_{\mathcal{T}}\phi_{\mathcal{T}}(F_{\mathcal{T}})=f$, which implies that $\mathfrak{I}_{p}(F_{\mathcal{T}})_{\mathcal{T}\in\mathcal{P}}=f$. This concludes the proof.
\end{proof}
It is appropriate to note the following in accordance with the proof of Theorem \ref{T1}. One is tempted to think that in order to prove the surjectivity of $\mathfrak{I}_{p}$, for each $f\in L^{p}(\mathcal{C}_{\mathbf{x}_{0}},\mu_{\mathbf{x}_{0}})$, it is sufficient to use Lemma \ref{L2.1}, choosing a sequence $(f_{\mathcal{T}})_{\mathcal{T}}\in \bigcup_{\mathcal{T}\in\mathcal{P}}\phi_{\mathcal{T}}(L^{p}(M^{\mathcal{T}},\mu_{\mathbf{x}_{0}}^{\mathcal{T}}))$ such that $\lim_{\mathcal{T}}\phi_{\mathcal{T}}(f_{\mathcal{T}})=f$. Nonetheless, it must be observed that the sequence $(f_{\mathcal{T}})_{\mathcal{T}}$ is not necessarily contained in the Cartesian product $\bigtimes_{\mathcal{T}\in\mathcal{P}}L^{p}(M^{\mathcal{T}},\mu_{\mathbf{x}_{0}}^{\mathcal{T}})$ which makes things a bit more involved.
\par We define the morphisms $\psi_{\mathcal{R}}:L^{p}(M^{\mathcal{R}},\mu^{\mathcal{R}}_{\mathbf{x}_{0}})\to \mathfrak{L}( L^{p}(M^{\mathcal{T}},\mu^{\mathcal{T}}_{\mathbf{x}_{0}}))$, $\mathcal{R}\in\mathcal{P}$, through 
\begin{equation}\label{An}
\psi_{\mathcal{R}}(f_{\mathcal{R}})=(g_{\mathcal{T}})_{\mathcal{T}\in\mathcal{P}}, \quad
g_{\mathcal{T}}:=\left\{
\begin{array}{ll}
\pi^{\ast}_{\mathcal{R}\mathcal{T}}(f_{\mathcal{R}}) & \text{ if } \mathcal{R}\subset\mathcal{T} \\
0 & \text{ else } 
\end{array}
\right.
\end{equation}
By Theorem \ref{T1}, the space $\mathfrak{L}( L^{p}(M^{\mathcal{T}},\mu^{\mathcal{T}}_{\mathbf{x}_{0}}))$ is a Banach space and therefore the pair $(\mathfrak{L}( L^{p}(M^{\mathcal{T}},\mu^{\mathcal{T}}_{\mathbf{x}_{0}})),\psi_{\mathcal{R}})$ defines a cocone.
The colimit of $(L^{p}(M^{\mathcal{T}},\mu^{\mathcal{T}}_{\mathbf{x}_{0}}),\pi^{\ast}_{\mathcal{T}\mathcal{T}'})$ will be identify with $(\mathfrak{L}( L^{p}(M^{\mathcal{T}},\mu^{\mathcal{T}}_{\mathbf{x}_{0}})),\psi_{\mathcal{R}})$. 

\begin{theorem}
	\label{JR}
	The cocone $(\mathfrak{L}( L^{p}(M^{\mathcal{T}},\mu^{\mathcal{T}}_{\mathbf{x}_{0}})),\psi_{\mathcal{T}})$ defines another relization of the colimit of $(L^{p}(M^{\mathcal{T}},\mu^{\mathcal{T}}_{\mathbf{x}_{0}}),\pi^{\ast}_{\mathcal{T}\mathcal{T}'})$.
\end{theorem}
\begin{proof}
	Since the colimit is unique up to a unique isomorphism on the category of cocones, it is enough to prove that $(\mathfrak{L}( L^{p}(M^{\mathcal{T}},\mu^{\mathcal{T}}_{\mathbf{x}_{0}})),\psi_{\mathcal{T}})$ is isomorphic to $(L^{p}(\mathcal{C}_{\mathbf{x}_{0}}(M),\mu_{\mathbf{x}_{0}}),\phi_{\mathcal{T}})$ in the category of cocones. By the proof of Theorem \ref{T1}, the map 
	\begin{equation*}
	\mathfrak{I}_{p}:\mathfrak{L}( L^{p}(M^{\mathcal{T}},\mu^{\mathcal{T}}_{\mathbf{x}_{0}}))\longrightarrow L^{p}(\mathcal{C}_{\mathbf{x}_{0}}(M),\mu_{\mathbf{x}_{0}}), \quad (f_{\mathcal{T}})_{\mathcal{T}\in\mathcal{P}}\mapsto \lim_{\mathcal{T}}\phi_{\mathcal{T}}(f_{\mathcal{T}}).
	\end{equation*}
	is an isometric isomorphism. To prove that it defines an isomorphism of cocones, we need to verify the commutativity of the diagram of Figure 2 for each $\mathcal{T}\in\mathcal{P}$.
	\begin{center}
	\begin{figure}[h!]
		\[
		\xymatrix@C+1em@R+1em{ 
			L^{p}(M^{\mathcal{T}},\mu^{\mathcal{T}}_{\mathbf{x}_{0}}) \ar^{\pi^{\ast}_{\mathcal{T}\mathcal{T}'}}[rr] \ar^{\psi_{\mathcal{T}}}[dr] \ar@/_1em/_{\phi_{\mathcal{T}}}[ddr] & & L^{p}(M^{\mathcal{T}'},\mu^{\mathcal{T}'}_{\mathbf{x}_{0}}) \ar_{\psi_{\mathcal{T}'}}[dl] \ar@/^1em/^{\phi_{\mathcal{T}'}}[ddl] \\
			&\mathfrak{L}( L^{p}(M^{\mathcal{T}},\mu^{\mathcal{T}}_{\mathbf{x}_{0}})) \ar^{\mathfrak{I}_{p}}[d] & \\
			&  L^{p}(\mathcal{C}_{\mathbf{x}_{0}}(M),\mu_{\mathbf{x}_{0}}) &
		}
		\]
		\label{F9}
		\caption{Diagram of the Universal Property II}
	\end{figure}
\end{center}
\noindent Take $f_{\mathcal{T}}\in L^{p}(M^{\mathcal{T}},\mu^{\mathcal{T}}_{\mathbf{x}_{0}})$, then by a simple computation we obtain
\begin{equation*}
(\mathfrak{I}_{p}\circ \psi_{\mathcal{T}})(f_{\mathcal{T}})=\lim_{\mathcal{Q}}h_{\mathcal{Q}}, \quad h_{\mathcal{Q}}:=\left\{
\begin{array}{ll}
\phi_{\mathcal{T}}(f_{\mathcal{T}}) & \text{ if } \mathcal{T}\subset\mathcal{Q} \\
0 & \text{ else } 
\end{array}
\right.
\end{equation*}
and thus $(\mathfrak{I}_{p}\circ \psi_{\mathcal{T}})(f_{\mathcal{T}})=\phi_{\mathcal{T}}(f_{\mathcal{T}})$. Hence the diagram of Figure 2 commutes and the proof is concluded.
\end{proof}
Here in after, as a direct consequence of Theorem \ref{JR}, we can and we shall denote $\mathfrak{L}( L^{p}(M^{\mathcal{T}},\mu^{\mathcal{T}}_{\mathbf{x}_{0}}))\equiv \underrightarrow{\lim} \ L^{p}(M^{\mathcal{T}},\mu^{\mathcal{T}}_{\mathbf{x}_{0}})$.
\par As a rather direct application of the isometric property of $\mathfrak{I}_{p}$, we obtain our integral limit approximation. Indeed, if $F\in L^{p}(\mathcal{C}_{\mathbf{x}_{0}}(M),\mu_{\mathbf{x}_{0}})$, there exists an element $(f_{\mathcal{T}})_{\mathcal{T}\in\mathcal{P}}\in\underrightarrow{\lim} \ L^{p}(M^{\mathcal{T}},\mu^{\mathcal{T}}_{\mathbf{x}_{0}})$ such that $\|F\|_{L^{p}}=\|(f_{\mathcal{T}})_{\mathcal{T}}\|_{\mathfrak{L}}$. Therefore, we have
	\begin{equation*}
	\int_{\mathcal{C}_{\mathbf{x}_{0}}(M)}|F|^{p}\ d\mu_{\mathbf{x}_{0}}=\lim_{\mathcal{T}}\int_{M^{\mathcal{T}}}|f_{\mathcal{T}}|^{p}\ d\mu_{\mathbf{x}_{0}}^{\mathcal{T}}.
	\end{equation*}
\noindent Furthermore, if we take into account that given $F\in L^{1}(\mathcal{C}_{\mathbf{x}_{0}}(M),\mu_{\mathbf{x}_{0}})$, we can write it as $F=F^{+}-F^{-}$ with $F^{+}, F^{-}$ positive and $F^{+},F^{-}\in L^{1}(\mathcal{C}_{\mathbf{x}_{0}}(M),\mu_{\mathbf{x}_{0}})$,  then we get the following result.

\begin{theorem}\label{C4.4}
	Let $F\in L^{1}(\mathcal{C}_{\mathbf{x}_{0}}(M),\mu_{\mathbf{x}_{0}})$, then there exists $(f_{\mathcal{T}})_{\mathcal{T}\in\mathcal{P}}\in\bigtimes_{\mathcal{T}\in\mathcal{P}}L^{1}(M^{\mathcal{T}},\mu_{\mathbf{x}_{0}}^{\mathcal{T}})$ such that 
	\begin{equation*}
	\int_{\mathcal{C}_{\mathbf{x}_{0}}(M)}F \ d\mu_{\mathbf{x}_{0}}=\lim_{\mathcal{T}}\int_{M^{\mathcal{T}}}f_{\mathcal{T}}\ d\mu_{\mathbf{x}_{0}}^{\mathcal{T}}.
	\end{equation*}
\end{theorem}


\section{Stratonovich Stochastic Integral}

In this section, we will apply the developed theory to a particular example, the Stratonovich stochastic integral. Our approach can be used to get a finite dimensional approximation scheme for this type of integrals, in contrast to the Andersson and Driver's framework \cite{1} that only can be used if the involved functional is bounded and continuous. 
\par Let us firstly recall briefly some basic facts about stochastic integration. Let $(X,Y)=(\{X_{t}\}_{t\in[0,1]},\{Y_{t}\}_{t\in[0,1]})$ be a pair of bounded $\mathbb{R}$-valued semimartingales defined in the probability space $(\Omega,\mathcal{F},\mu)$. Then, the Stratonovich integral of $X$ with respect to $Y$ is defined by the relation
\begin{equation*}
	\int_{0}^{1}X_{t}\circ dY_{t}:=\lim_{L^{2}(\mu)}\sum_{i=1}^{n}\frac{X_{t^{n}_{i}}+X_{t^{n}_{i-1}}}{2}(Y_{t^{n}_{i}}-Y_{t^{n}_{i-1}})\in L^{2}(\Omega,\mu)
\end{equation*}
where $\mathcal{T}=\{\mathcal{T}_{n}\}_{n\in\mathbb{N}}$, $\mathcal{T}_{n}:=\{t^{n}_{i}\}_{i=0}^{n}$, is a fixed partition with mesh zero of $[0,1]$. It is related to the It\^{o} stochastic integral by the relation
\begin{equation*}
\int_{0}^{1}X_{t}\circ dY_{t}=\int_{0}^{1}X_{t} \ dY_{t}+[X,Y]_{t}
\end{equation*}
where $[X,Y]_{t}$ denotes the covariation of the processes $(X,Y)$ and $dY_{t}$ denotes the It\^{o} differential.
In the case in which $(X,Y)=(\{X_{t}\}_{t\in[0,1]},\{Y_{t}\}_{t\in[0,1]})$ are bounded $\mathbb{R}^{N}$-valued semimartingales, we define
\begin{equation*}
\int_{0}^{1}X_{t}\circ dY_{t}:=\sum_{i=1}^{N}\int_{0}^{1}X^{i}_{t}\circ dY^{i}_{t}.
\end{equation*}
It is worth to mention that the usual definition of the Stratonovich integral is under convergence in probability \cite[Th. 26, Ch. V]{P}. Since we will deal with semimartingales defined on compact manifolds, we only need the definition for bounded ones, in which case, the convergence in probability implies the $L^{2}$-convergence as a consequence of Vitali's convergence Theorem.
\par It must be taken into account that if $(M,g)$ is a closed Riemannian manifold embedded in the Euclidean space $\mathbb{R}^{N}$, by \cite[Pr. 3.2.1]{S}, the $M$-valued stochastic process $X=\{X_{t}\}_{t\in[0,1]}$ defined by the coordinate functionals of $(\mathcal{C}_{\mathbf{x}_{0}}(M),\mu_{\mathbf{x}_{0}})$, defines a $\mathbb{R}^{N}$-valued bounded semimartingale. This implies by \cite[Pr. 1.2.7, (i)]{S} that $\{f(X_{t})\}_{t\in[0,1]}$ is a real valued semimartingale for each $f\in\mathcal{C}^{\infty}(M)$. Therefore the  Stratonovich stochastic integral of $f(X)$ with respect to $X$, where $f\in\mathcal{C}^{\infty}(M,\mathbb{R}^{N})$, $f=(f_{1},f_{2},...,f_{N})$, is well defined. 
\par The main result of this section is the following. It establishes which is exactly the preimage of the functional $\int_{0}^{1}f(X_{t})\circ dX_{t}\in L^{2}(\mathcal{C}_{\mathbf{x}_{0}}(M),\mu_{\mathbf{x}_{0}})$ via the identification $\mathfrak{I}_{2}:\underrightarrow{\lim} \ L^{2}(M^{\mathcal{T}},\mu_{\mathbf{x}_{0}}^{\mathcal{T}})\to L^{2}(\mathcal{C}_{\mathbf{x}_{0}}(M),\mu_{\mathbf{x}_{0}})$.

\begin{theorem}\label{th4.1}
	Let $(M,g)$ be a closed Riemannian manifold embedded in the Euclidean space $\mathbb{R}^{N}$, $f\in\mathcal{C}^{\infty}(M,\mathbb{R}^{N})$, $f=(f_{1},f_{2},...,f_{N})$, and $\{X_{t}\}_{t\in[0,1]}$ the $M$-valued semimartingale defined by the coordinate functionals of $(\mathcal{C}_{\mathbf{x}_{0}}(M),\mu_{\mathbf{x}_{0}})$. Then
	\begin{equation*}
	\mathfrak{I}_{2}\left(\sum_{j=1}^{N}\sum_{i=1}^{n}\frac{f_{j}(\mathbf{x}_{t_{i}})+f_{j}(\mathbf{x}_{t_{i-1}})}{2}(x^{j}_{t_{i}}-x^{j}_{t_{i-1}})\right)_{\mathcal{T}\in\mathcal{P}}=\int_{0}^{1}f(X_{t})\circ dX_{t},
	\end{equation*}
	where $\mathcal{T}=\{0=t_{0}<t_{1}<\cdots<t_{n}\}$ and $\mathbf{x}_{t_{i}}=(x^{1}_{t_{i}},\cdots,x^{N}_{t_{i}})\in M\subset \mathbb{R}^{N}$.
\end{theorem}
\begin{proof}
	Since $M$ is compact and embedded in $\mathbb{R}^{N}$, the process $X=\{X_{t}\}_{t\in[0,1]}$ is, in particular, a bounded $\mathbb{R}^{N}$-valued semimartingale and its Stratonovich integral is defined in the usual manner by
	\begin{align*}
	\int_{0}^{1}f(X_{t})\circ dX_{t}& :=\sum_{j=1}^{N}\int_{0}^{1}f_{j}(X_{t})\circ dX^{j}_{t}\\
	& =\lim_{L^{2}(\mu_{\mathbf{x}_{0}})}\sum_{j=1}^{N}\sum_{i=1}^{n}\frac{f_{j}(X_{t^{n}_{i}})+f_{j}(X_{t^{n}_{i-1}})}{2}(X^{j}_{t^{n}_{i}}-X^{j}_{t^{n}_{i-1}}),
	\end{align*}
	for every partition $\mathcal{Q}=\{\mathcal{Q}_{n}\}_{n\in\mathbb{N}}$, $\mathcal{Q}_{n}:=\{t^{n}_{i}\}_{i=0}^{n}$ with mesh zero of $[0,1]$ satisfying $t^{n}_{0}=0$ for each $n\in\mathbb{N}$. This limit can be expressed as a convergence of nets via
	\begin{equation*}
	\int_{0}^{1}f(X_{t})\circ dX_{t}=\lim_{\mathcal{T}}\sum_{j=1}^{N}\sum_{i=1}^{n}\frac{f_{j}(X_{t_{i}})+f_{j}(X_{t_{i-1}})}{2}(X^{j}_{t_{i}}-X^{j}_{t_{i-1}}),
	\end{equation*}
	where $\mathcal{T}=\{0=t_{0}<t_{1}<\cdots<t_{n}\}$. The limit is understood in $L^{2}$-convergence. On the other hand, we define the maps $F^{j}_{\mathcal{T}}: M^{\mathcal{T}} \to \mathbb{R}$, $j\in \{1,2,...,N\}$, by
	\begin{equation*}
	F^{j}_{\mathcal{T}}(\mathbf{x}_{t_{i}})_{i=1}^{n}:=\sum_{i=1}^{n} \frac{f_{j}(\mathbf{x}_{t_{i}})+f_{j}(\mathbf{x}_{t_{i-1}})}{2}(x^{j}_{t_{i}}-x^{j}_{t_{i-1}}),
	\end{equation*}
	where $\mathcal{T}=\{0=t_{0}<t_{1}<\cdots<t_{n}\}$ and $\mathbf{x}_{t_{i}}=(x^{1}_{t_{i}},x^{2}_{t_{i}},\cdots,x^{n}_{t_{i}})$. Then, by the boundedness of each $F^{j}_{\mathcal{T}}$, it follows that $F^{j}_{\mathcal{T}}\in L^{2}(M^{\mathcal{T}},\mu^{\mathcal{T}}_{\mathbf{x}_{0}})$ and by the definition of the morphisms $\phi_{\mathcal{T}}:L^{2}(M^{\mathcal{T}},\mu_{\mathbf{x}_{0}}^{\mathcal{T}})\to L^{2}(\mathcal{C}_{\mathbf{x}_{0}}(M),\mu_{\mathbf{x}_{0}})$, we obtain $$\phi_{\mathcal{T}}(F^{j}_{\mathcal{T}})=\sum_{i=1}^{n}\frac{f_{j}(X_{t_{i}})+f_{j}(X_{t_{i-1}})}{2}(X^{j}_{t_{i}}-X^{j}_{t_{i-1}}).$$
	Therefore, if we prove that $(F^{j}_{\mathcal{T}})_{\mathcal{T}\in\mathcal{P}}\in \underrightarrow{\lim} \ L^{2}(M^{\mathcal{T}},\mu_{\mathbf{x}_{0}}^{\mathcal{T}})$ for each $j\in\{1,2,...,N\}$, by the definition of $\mathfrak{I}_{2}$ and the linearity of $\phi_{\mathcal{T}}$, we will deduce 
	\begin{align*}
	\mathfrak{I}_{2}\left(\sum_{j=1}^{N}F^{j}_{\mathcal{T}}\right)_{\mathcal{T}\in\mathcal{P}}&=\lim_{\mathcal{T}}\sum_{j=1}^{N}\phi_{\mathcal{T}}(F^{j}_{\mathcal{T}})\\ 
	& =\lim_{\mathcal{T}}\sum_{j=1}^{N}\sum_{i=1}^{n}\frac{f_{j}(X_{t_{i}})+f_{j}(X_{t_{i-1}})}{2}(X^{j}_{t_{i}}-X^{j}_{t_{i-1}})\\
	& =\int_{0}^{1}f(X_{t})\circ dX_{t},
	\end{align*}
	obtaining the required result. Finally, we prove that, indeed, $(F^{j}_{\mathcal{T}})_{\mathcal{T}\in\mathcal{P}}\in \underrightarrow{\lim} \ L^{2}(M^{\mathcal{T}},\mu_{\mathbf{x}_{0}}^{\mathcal{T}})$ for each $j\in\{1,2,...,N\}$. Since each $\phi_{\mathcal{T}}$ is an isometric isomorphism, we have
	\begin{align*}
	\|\pi_{\mathcal{R}\mathcal{T}}^{\ast}(F^{j}_{\mathcal{R}})-F^{j}_{\mathcal{T}}\|_{L^{2}_{\mathcal{T}}}&=\|(\phi_{\mathcal{T}}\circ\pi_{\mathcal{R}\mathcal{T}}^{\ast})(F^{j}_{\mathcal{R}})-\phi_{\mathcal{T}}(F^{j}_{\mathcal{T}})\|_{L^{2}}\\
	&=\|\phi_{\mathcal{R}}(F^{j}_{\mathcal{R}})-\phi_{\mathcal{T}}(F^{j}_{\mathcal{T}})\|_{L^{2}}, \quad \mathcal{R}\subset \mathcal{T}.
	\end{align*}
	The convergence of $\{\phi_{\mathcal{T}}(F^{j}_{\mathcal{T}})\}_{\mathcal{T}\in\mathcal{P}}$ (that is justified by the existence of the integral $\int_{0}^{1}f(X_{t})\circ dX_{t}$) together with the last identity, implies the co-Cauchy property for $(F^{j}_{\mathcal{T}})_{\mathcal{T}\in\mathcal{P}}$. Hence $(F^{j}_{\mathcal{T}})_{\mathcal{T}\in\mathcal{P}}\in \underrightarrow{\lim} \ L^{2}(M^{\mathcal{T}},\mu_{\mathbf{x}_{0}}^{\mathcal{T}})$ and the proof is concluded.
\end{proof}

As a direct consequence of the preceding result and the isometric property of $\mathfrak{I}_{2}$, we deduce the following finite dimensional approximation under the same hypothesis of Theorem \ref{th4.1},
\begin{equation*}
\int_{\mathcal{C}_{\mathbf{x}_{0}}(M)}\left| \int_{0}^{1} f(X_{t})\circ dX_{t}\right|^{2} d\mu_{\mathbf{x}_{0}}=\lim_{\mathcal{T}}\int_{M^{\mathcal{T}}}\left|\sum_{j=1}^{N}\sum_{i=1}^{n}\frac{f_{j}(\mathbf{x}_{t_{i}})+f_{j}(\mathbf{x}_{t_{i-1}})}{2}(x^{j}_{t_{i}}-x^{j}_{t_{i-1}})\right|^{2} d\mu^{\mathcal{T}}_{\mathbf{x}_{0}},
\end{equation*}
where $\mathcal{T}=\{0=t_{0}<t_{1}<\cdots<t_{n}\}$. Finally, it is convenient to remark that all the considerations made in this section can be adapted easily to the cover the It\^{o} stochastic integration.

\section{Approximation Scheme for the Geometric Framework}

In this final section, we will give an approximation scheme based in the geometric measure introduced by Andersson and Driver in \cite{1}. Up to here, we have proved that the spaces $L^{p}(\mathcal{C}_{\mathbf{x}_{0}}(M),\mu_{\mathbf{x}_{0}})$, $1\leq p<\infty$, define a realization of the colimit of the diagram $(L^{p}(M^{\mathcal{T}},\mu_{\mathbf{x}_{0}}^{\mathcal{T}}),\pi^{\ast}_{\mathcal{T}\mathcal{T}'})$, where the measures $\mu_{\mathbf{x}_{0}}^{\mathcal{T}}$ are given by
\begin{equation*}
d\mu_{\mathbf{x}_{0}}^{\mathcal{T}}=\prod_{i=1}^{n}p_{t_{i}-t_{i-1}}(x_{t_{i}},x_{t_{i-1}})\prod_{i=1}^{n}d\mu(x_{t_{i}}), \quad \mathcal{T}=\{0=t_{0}<t_{1}<\cdots<t_{n}\}.
\end{equation*}
Moreover, we have provided a particular realization, $\underrightarrow{\lim} \ L^{p}(M^{\mathcal{T}},\mu_{\mathbf{x}_{0}}^{\mathcal{T}})$, of the colimit that allows to stablish a finite dimensional approximation result for integrable functionals.
\par However, observe that this approximation scheme does not follows the philosophy of Andersson and Driver approach materialized in equation \eqref{E2}. For this reason, in this section, we adapt our scheme to cover the measure considered by Andersson and Driven in \cite{1}. We start recalling some facts and definitions of \cite{1}. Through this section, the notation $\mathcal{C}_{0}(\mathbb{R}^{m})$ stands for the space of continuous paths $\gamma:[0,1]\to\mathbb{R}^{m}$ such that $\gamma(0)=0$.

\subsection{Piecewise Linear Path Space} Consider $H(\mathbb{R}^{m})\subset \mathcal{C}_{0}(\mathbb{R}^{m})$ the subspace consisting on finite energy paths
\begin{equation*}
H(\mathbb{R}^{m}):=\{\gamma\in\mathcal{C}_{0}(\mathbb{R}^{m}): \gamma \text{ is absolutely continuous and } E_{\mathbb{R}^{m}}(\gamma)<\infty\},
\end{equation*}
where the energy functional is defined by
\begin{equation*}
E_{\mathbb{R}^{m}}(\gamma):=\int_{0}^{1}\langle \gamma'(s),\gamma'(s)\rangle \ ds.
\end{equation*}
For each partition $\mathcal{T}=\{0=t_{0}<t_{1}<\cdots<t_{n}=1\}$, the space of piecewise linear paths on $\mathbb{R}^{m}$ with respect to $\mathcal{T}$, subsequently denoted by $H_{\mathcal{T}}(\mathbb{R}^{m})$, is defined through
\begin{equation*}
H_{\mathcal{T}}(\mathbb{R}^{m}):=\{\gamma\in\mathcal{C}_{0}(\mathbb{R}^{m}): \gamma \text{ is linear for } t\notin \mathcal{T}\}.
\end{equation*}
Clearly $H_{\mathcal{T}}(\mathbb{R}^{m})$ is linearly isomorphic to $\mathbb{R}^{m\times \mathcal{T}}$, where
$$\mathbb{R}^{m\times\mathcal{T}}:=\bigtimes_{t\in\mathcal{T}\backslash\{0\}}\mathbb{R}^{m},$$
via the linear isomorphism
\begin{equation*}
\Pi_{\mathcal{T}}:H_{\mathcal{T}}(\mathbb{R}^{m})\longrightarrow \mathbb{R}^{m\times \mathcal{T}}, \quad \Pi_{\mathcal{T}}(\gamma)=(\gamma(t_{1}),\gamma(t_{2}),...,\gamma(t_{n})).
\end{equation*}
Since $H_{\mathcal{T}}(\mathbb{R}^{m})$ is linear, it follows that $T_{\gamma}H_{\mathcal{T}}(\mathbb{R}^{m})\simeq H_{\mathcal{T}}(\mathbb{R}^{m})$ for each $\gamma\in H_{\mathcal{T}}(\mathbb{R}^{m})$. We introduce the Riemmanian metric on $H_{\mathcal{T}}(\mathbb{R}^{m})$, $h_{\mathcal{T}}\in\Gamma(T^{*}H_{\mathcal{T}}(\mathbb{R}^{m})\otimes T^{*}H_{\mathcal{T}}(\mathbb{R}^{m}))$, defined by
\begin{equation*}
h_{\mathcal{T}}(u,v):=\int_{0}^{1}\langle u'(s), v'(s)\rangle \ ds, \quad u,v\in T_{\gamma}H_{\mathcal{T}}(\mathbb{R}^{m}), \ \gamma\in H_{\mathcal{T}}(\mathbb{R}^{m}),
\end{equation*}
and its corresponding volume form $\Vol_{h_{\mathcal{T}}}\in\Gamma(\wedge^{m\times n}TH_{\mathcal{T}}(\mathbb{R}^{m}))$ determinated by 
\begin{equation*}
\Vol_{h_{\mathcal{T}}}(u_{1},u_{2},...,u_{m\times n}):=\sqrt{\det(h_{\mathcal{T}}(u_{i},u_{j})_{ij})},
\end{equation*}
where $\{u_{1},u_{2},...,u_{m\times n}\}\subset T_{\gamma}H_{\mathcal{T}}(\mathbb{R}^{m})$ is an oriented basis and $\gamma\in H_{\mathcal{T}}(\mathbb{R}^{m})$. Finally, we introduce a Borel measure $\mu_{\mathcal{T}}$ on $H_{\mathcal{T}}(\mathbb{R}^{m})$. As usual, when we deal with a measure associated with a Riemannian structure, we choose the Borel $\sigma$-algebra associated with the topology induced by the corresponding Riemannian metric. Along this section, this will be done several times without specifying it again.

\begin{definition}
For each partition $\mathcal{T}=\{0=t_{0}<t_{1}<\cdots<t_{n}=1\}$ of $[0,1]$, we denote by $\mu_{\mathcal{T}}$ the Borel measure on $H_{\mathcal{T}}(\mathbb{R}^{m})$ defined by the density
\begin{equation*}
d\mu_{\mathcal{T}}=\frac{1}{(\sqrt{2 \pi})^{mn}}\exp\left\{-\frac{1}{2}E_{\mathbb{R}^{m}}\right\}\Vol_{h_{\mathcal{T}}}.
\end{equation*}
\end{definition}

\subsection{Piecewise Geodesic Path Space} Now, we define the curved analogue of the measure space $(H_{\mathcal{T}}(\mathbb{R}^{m}),\mu_{\mathcal{T}})$. Let $(M,g)$ be a closed Riemannian manifold of dimension $m$. Consider $H(M)\subset \mathcal{C}_{\mathbf{x}_{0}}(M)$ to be the Hilbert manifold of finite energy paths, defined by
\begin{equation*}
H(M):=\{\gamma\in\mathcal{C}_{\mathbf{x}_{0}}(M): \gamma \text{ is absolutely continuous and } E(\gamma)<\infty\}
\end{equation*}
where the energy functional $E$ is given through
\begin{equation*}
E(\gamma):=\int_{0}^{1}g(\gamma'(s),\gamma'(s)) \ ds.
\end{equation*}
Recall that $\gamma\in \mathcal{C}_{\mathbf{x}_{0}}(M)$ is said to be absolutely continuous if $f\circ \gamma$ is absolutely continuous for all $f\in\mathcal{C}^{\infty}(M)$. The tangent space $T_{\gamma}H(M)$ to $H(M)$ at $\gamma$ can be identified with the space of absolutely continuous vector fields $X:[0,1]\to TM$ along $\gamma$ such that $X(0)=0$ and $G^{1}(X,X)<\infty$ where
\begin{equation*}
G^{1}(X,X):=\int_{0}^{1}g\left(\frac{\nabla X(t)}{dt},\frac{\nabla X(t)}{dt}\right) \ dt.
\end{equation*}
As usual, we denote
\begin{equation*}
\frac{\nabla X(t)}{dt}:=\slash\slash_{t}(\gamma)\frac{d}{dt}\{\slash\slash_{t}(\gamma)^{-1} X(t)\},
\end{equation*}
where $\slash\slash_{t}(\gamma):T_{\mathbf{x}_{0}}M\to T_{\gamma(\mathbf{x}_{0})}M$ denotes the parallel translation along $\gamma$ relative to the Levy-Civita covariant derivative $\nabla$.
\par Let $\mathcal{T}=\{0=t_{0}<t_{1}<\cdots<t_{n}=1\}$ be a partition of $[0,1]$. We define the subspace of $H(M)$,
\begin{equation*}
H_{\mathcal{T}}(M):=\left\{\gamma\in H(M)\cap \mathcal{C}^{2}([0,1]\backslash\mathcal{T},M): \nabla \gamma'(t)/dt=0 \text{ for } t\notin\mathcal{T} \right\},
\end{equation*}
consisting on piecewise geodesic paths in $H(M)$ which change directions only at the partition points. The space $H_{\mathcal{T}}(M)$ is a finite dimensional submanifold of $H(M)$ of dimension $n\times m$. For $\gamma\in H_{\mathcal{T}}(M)$, the tangent space $T_{\gamma}H_{\mathcal{T}}(M)$ can be identified with elements $X\in T_{\gamma}H(M)$ satisfying the Jacobi equations on $[0,1]\backslash \mathcal{T}$. In other words, $X\in T_{\gamma}H(M)$ is in $T_{\gamma}H_{\mathcal{T}}(M)$ if and only if 
\begin{equation*}
\frac{\nabla^{2}}{dt^{2}}X(t)=R(\gamma'(t),X(t))\gamma'(t),
\end{equation*}
where $R$ is the curvature tensor of $\nabla$. We can give to $H_{\mathcal{T}}(M)$ a Riemannian structure introducing the $\mathcal{T}$-metric. The $\mathcal{T}$-metric $g_{\mathcal{T}}\in\Gamma(T^{*}H_{\mathcal{T}}(M)\otimes T^{*}H_{\mathcal{T}}(M))$ is defined by
\begin{equation*}
g_{\mathcal{T}}(X,Y):=\sum_{i=1}^{n}g\left( \frac{\nabla X(t_{i-1}+)}{dt}, \frac{\nabla Y(t_{i-1}+)}{dt} \right) \Delta_{i}t, \quad X,Y\in T_{\gamma}H_{\mathcal{T}}(M),
\end{equation*}
for each $\gamma\in H_{\mathcal{T}}(M)$, where the notation $\nabla X(t_{i-1}+) / dt$ is a shorthand of $\lim_{t\downarrow t_{i-1}} \nabla X(t)/dt$. We denote by $\Vol_{g_{\mathcal{T}}}\in \Gamma(\wedge^{m\times m}TH_{\mathcal{T}}(M))$ the volume form associated to $g_{\mathcal{T}}$. It is determinated by
$$ \Vol_{g_{\mathcal{T}}}(X_{1},X_{2},...,X_{m\times n}):=\sqrt{\det(g_{\mathcal{T}}(X_{i},X_{j})_{ij})},$$
where $\{X_{1}, X_{2},...,X_{m\times n}\}\subset T_{\gamma}H_{\mathcal{T}}(M)$ is an oriented basis and $\gamma\in H_{\mathcal{T}}(M)$.
\begin{definition}
	For each partition $\mathcal{T}=\{0=t_{0}<t_{1}<\cdots<t_{n}=1\}$ of $[0,1]$, we denote by $\nu_{\mathcal{T}}$ the Borel measure on $H_{\mathcal{T}}(M)$ defined by the density
	\begin{equation*}
	d\nu_{\mathcal{T}}:=\frac{1}{(\sqrt{2 \pi})^{mn}}\exp\left\{-\frac{1}{2}E\right\}\Vol_{g_{\mathcal{T}}}.
	\end{equation*}
\end{definition}

This measure spaces $(H_{\mathcal{T}}(M),\mathcal{B}_{\mathcal{T}},\nu_{\mathcal{T}})$, where $\mathcal{B}_{\mathcal{T}}$ stands for the Borel $\sigma$-algebra of $H_{\mathcal{T}}(M)$, will be the finite dimensional candidates for the approximation scheme.

\subsection{Cartan's Development map} In general, it is not quite easy to deal with the manifold $H_{\mathcal{T}}(M)$. Due to this fact, we will identify this space through the well known space $H_{\mathcal{T}}(\mathbb{R}^{m})$.
The cited identification is done via the Cartan's development map. The Cartan's development map $\Phi: H(\mathbb{R}^{m})\to H(M)$ is defined, for $\alpha\in H(\mathbb{R}^{m})$, by $\Phi(\alpha):=\gamma$, where $\gamma\in H(M)$ is the unique solution of the ordinary differential equation 
\begin{equation}\label{E0}
\gamma'(t)=\slash\slash_{t} (\gamma) \alpha'(t), \quad \gamma(0)=\mathbf{x}_{0}.
\end{equation}
The anti-development map $\Phi^{-1}:H(M)\to H(\mathbb{R}^{m})$ is defined by $\Phi^{-1}(\gamma):=\alpha$ where $\alpha\in H(\mathbb{R}^{m})$ is given by
\begin{equation*}
\alpha(t):=\int_{0}^{t}\slash\slash_{r}^{-1}(\gamma)\gamma'(r) \ dr.
\end{equation*}
The map $\Phi:H(\mathbb{R}^{m})\to H(M)$ is bijective and smooth, hence it defines a diffeomorphism of infinite dimensional Hilbert manifolds, but it is not in general an isometry of Riemannian manifolds. The development map $\Phi:H(\mathbb{R}^{m})\to H(M)$ has the property
\begin{equation*}
\Phi(H_{\mathcal{T}}(\mathbb{R}^{m}))=H_{\mathcal{T}}(M).
\end{equation*}
We shall denote $\Phi|_{H_{\mathcal{T}}(\mathbb{R}^{m})}$ by $\Phi_{\mathcal{T}}$. 
\par A fundamental property of the Cartan's development map is that it preserves the $\mathcal{T}$-measure, in the sense that
\begin{equation}\label{I1}
\mu_{\mathcal{T}}(B)=\nu_{\mathcal{T}}(\Phi_{\mathcal{T}}(B)), 
\end{equation}
for each Borel subset $B$ of $H_{\mathcal{T}}(\mathbb{R}^{m})$. 
\par Moreover, it can be seen that the development map relates the measure $\nu_{\mathcal{T}}$ of $H_{\mathcal{T}}(M)$ with the well known heat kernel measure in the flat space $\mathbb{R}^{m\times \mathcal{T}}$. For a partition $\mathcal{T}=\{0=t_{0}<t_{1}<\cdots<t_{n}=1\}$, the heat kernel measure $\lambda_{0}^{\mathcal{T}}$ is the Borel measure on $\mathbb{R}^{m\times\mathcal{T}}$ defined by
\begin{equation*}
d\lambda_{0}^{\mathcal{T}}:=\prod_{i=1}^{n}p_{t_{i}-t_{i-1}}(x_{t_{i}},x_{t_{i-1}})\prod_{i=1}^{n}dx_{t_{i}},
\end{equation*}
where $x_{0}=0$ and $p_{t}(x,y)$ is the heat kernel of $\mathbb{R}^{m}$. In \cite[Lem. 4.11]{1}, it is proved the identity
\begin{equation}\label{I2}
\mu_{\mathcal{T}}(\Pi_{\mathcal{T}}^{-1}(B))=\lambda_{0}^{\mathcal{T}}(B)
\end{equation}
for each Borel subset $B$ of $\mathbb{R}^{m\times\mathcal{T}}$. Hence joining equations \eqref{I1} and \eqref{I2} yields
\begin{equation*}
\nu_{\mathcal{T}}((\Phi_{\mathcal{T}}\circ \Pi_{\mathcal{T}}^{-1})(B))=\lambda_{0}^{\mathcal{T}}(B).
\end{equation*}
As a consequence, we can reduce the structure of the space $L^{p}\left(H_{\mathcal{T}}(M),\nu_{\mathcal{T}}\right)$ to the more familiar one $L^{p}(\mathbb{R}^{m\times \mathcal{T}},\lambda^{\mathcal{T}}_{0})$. That is, for $1\leq p < \infty$, the following spaces are identified
\begin{equation*}
L^{p}\left(H_{\mathcal{T}}(M),\nu_{\mathcal{T}}\right)\simeq L^{p}(\mathbb{R}^{m\times \mathcal{T}},\lambda^{\mathcal{T}}_{0}),
\end{equation*}
via the isometric isomorphism 
\begin{equation}\label{EQG}
\Lambda_{\mathcal{T}}:L^{p}\left(H_{\mathcal{T}}(M),\nu_{\mathcal{T}}\right)\longrightarrow L^{p}(\mathbb{R}^{m\times \mathcal{T}},\lambda^{\mathcal{T}}_{0}), \quad \Lambda_{\mathcal{T}}(f):=f\circ\Phi_{\mathcal{T}}\circ\Pi_{\mathcal{T}}^{-1}.
\end{equation}
It can be also introduced an almost everywhere defined stochastic extension of the development map, called the stochastic development map $\tilde{\Phi}:\mathcal{C}_{0}(\mathbb{R}^{m})\to\mathcal{C}_{\mathbf{x}_{0}}(M)$ and its corresponding anti-development map $\tilde{\Phi}^{-1}:\mathcal{C}_{\mathbf{x}_{0}}(M)\to\mathcal{C}_{0}(\mathbb{R}^{m})$. It can be proved that they preserve the Wiener measure in the sense that
\begin{equation*}
\lambda_{0}(B)=\mu_{\mathbf{x}_{0}}(\tilde{\Phi}(B)),
\end{equation*}
for each Borel subset $B$ of $\mathcal{C}_{0}(\mathbb{R}^{m})$ where $\lambda_{0}$ is the Wiener measure on the classical Wiener space $\mathcal{C}_{0}(\mathbb{R}^{m})$ (see for instance \cite{J}). Hence, thanks to the stochastic development map, we can reduce the structure of the space $L^{p}(\mathcal{C}_{\mathbf{x}_{0}}(M),\mu_{\mathbf{x}_{0}})$ to the more familiar one $L^{p}(\mathcal{C}_{0}(\mathbb{R}^{m}),\lambda_{0})$ under the philosophy of \eqref{EQG}. For $1\leq p<\infty$, the following spaces are identified
\begin{equation*}
L^{p}(\mathcal{C}_{\mathbf{x}_{0}}(M),\mu_{\mathbf{x}_{0}})\simeq L^{p}(\mathcal{C}_{0}(\mathbb{R}^{m}),\lambda_{0}),
\end{equation*}
via the map
\begin{equation}\label{P7.3}
\Lambda:L^{p}(\mathcal{C}_{\mathbf{x}_{0}}(M),\mu_{\mathbf{x}_{0}})\longrightarrow L^{p}(\mathcal{C}_{0}(\mathbb{R}^{m}),\lambda_{0}), \quad \Lambda(f):=f\circ\tilde{\Phi}.
\end{equation}

\subsection{Finite Dimensional Approximation Scheme} To prove an analogue of Theorem \ref{th4.7} for the geometric measure $\nu_{\mathcal{T}}$, we will follow the following philosophy. Firstly we shall prove the approximation scheme for the classical Wiener measure space $(\mathcal{C}_{0}(\mathbb{R}^{m}),\lambda_{0})$ and then we will make use the identifications provided by equations \eqref{EQG} and \eqref{P7.3}, to translate the result to the geometric framework.
\par Let $\mathcal{P}$ be the directed set consisting on partitions of $[0,1]$ $$\mathcal{T}=\{0=t_{0}<t_{1}<\cdots<t_{n}=1\},$$ partially ordered by inclusion and consider the projectors $$\pi_{\mathcal{T}\mathcal{T}'}:\mathbb{R}^{m\times\mathcal{T}'}\longrightarrow \mathbb{R}^{m\times\mathcal{T}}, \quad \pi_{\mathcal{T}\mathcal{T}'}(x_{t})_{t\in\mathcal{T}'\backslash\{0\}}:=(x_{t})_{t\in\mathcal{T}\backslash\{0\}}$$ $$\pi_{\mathcal{T}}:\mathbb{R}^{m\times [0,1]}\longrightarrow\mathbb{R}^{m\times\mathcal{T}}, \quad \pi_{\mathcal{T}}(x_{t})_{t\in[0,1]}:=(x_{t})_{t\in\mathcal{T}\backslash\{0\}}.$$ In analogy with the preceding sections, we define the diagram $(L^{p}(\mathbb{R}^{m\times \mathcal{T}},\lambda_{0}^{\mathcal{T}}),\eta_{\mathcal{T}\mathcal{T}'})$ where the morphisms $\eta_{\mathcal{T}\mathcal{T}'}$, $\mathcal{T}\subset \mathcal{T}'$, are defined through
\begin{equation*}
\eta_{\mathcal{T}\mathcal{T}'}:L^{p}(\mathbb{R}^{m\times\mathcal{T}},\lambda_{0}^{\mathcal{T}})\longrightarrow L^{p}(\mathbb{R}^{m\times \mathcal{T}'},\lambda_{0}^{\mathcal{T}'}), \quad \eta_{\mathcal{T}\mathcal{T}'}(f):=f\circ\pi_{\mathcal{T}\mathcal{T}'}.
\end{equation*}
Using the same techniques of section 4, it is easily proved that for each $1\leq p<\infty$, the cocone $(L^{p}(\mathcal{C}_{0}(\mathbb{R}^{m}),\lambda_{0}),\varphi_{\mathcal{T}})$ defines a realization of the colimit of $(L^{p}(\mathbb{R}^{m\times \mathcal{T}},\lambda_{0}^{\mathcal{T}}),\eta_{\mathcal{T}\mathcal{T}'})$, where the morphisms $\varphi_{\mathcal{T}}$, $\mathcal{T}\in\mathcal{P}$, are given by
\begin{equation*}
\varphi_{\mathcal{T}}:L^{p}(\mathbb{R}^{m\times\mathcal{T}},\lambda_{0}^{\mathcal{T}})\longrightarrow L^{p}(\mathcal{C}_{0}(\mathbb{R}^{m}),\lambda_{0}), \quad \varphi_{\mathcal{T}}(f):=f\circ\pi_{\mathcal{T}}|_{\mathcal{C}_{\mathbf{x}_{0}}(\mathbb{R}^{m})}.
\end{equation*}
This establishes the approximation scheme for the classical Wiener space $(\mathcal{C}_{0}(\mathbb{R}^{m}),\lambda_{0})$. Now, we proceed to define the diagram and cocone for the geometric measure $\nu_{\mathcal{T}}$. 
\par Let us consider the diagram $(L^{p}\left(H_{\mathcal{T}}(M),\nu_{\mathcal{T}}\right),\varrho_{\mathcal{T}\mathcal{T}'})$ where for each $\mathcal{T}\subset\mathcal{T}'$, $\varrho_{\mathcal{T}\mathcal{T}'}$ are the morphisms defined by
\begin{equation*}
\varrho_{\mathcal{T}\mathcal{T}'}:L^{p}\left(H_{\mathcal{T}}(M),\nu_{\mathcal{T}}\right)\longrightarrow L^{p}\left(H_{\mathcal{T}'}(M),\nu_{\mathcal{T}'}\right),\quad \varrho_{\mathcal{T}\mathcal{T}'}:=\Lambda_{\mathcal{T}'}^{-1}\circ \eta_{\mathcal{T}\mathcal{T}'}\circ\Lambda_{\mathcal{T}},
\end{equation*}
in other words, the unique morphisms making the diagram of Figure 3 commutative.
\begin{figure}[h!]
	\[
\xymatrix{ 
	L^{p}\left(H_{\mathcal{T}}(M),\nu_{\mathcal{T}}\right) \ar@{-->}^{\varrho_{\mathcal{T}\mathcal{T}'}}[d]  \ar^{\Lambda_{\mathcal{T}}}[r]  &  L^{p}(\mathbb{R}^{m\times \mathcal{T}},\lambda^{\mathcal{T}}_{0}) \ar^{\eta_{\mathcal{T}\mathcal{T}'}}[d] \\
	L^{p}\left(H_{\mathcal{T}'}(M),\nu_{\mathcal{T}'}\right) \ar^{\Lambda_{\mathcal{T}'}}[r] & L^{p}(\mathbb{R}^{m\times \mathcal{T}'},\lambda^{\mathcal{T}'}_{0})
}
\]
\label{F3}
\caption{Definition of the morphisms $\varrho_{\mathcal{T}\mathcal{T}'}$}
\end{figure}

The main result of this subsection is to prove that the cocone $(L^{p}(\mathcal{C}_{\mathbf{x}_{0}}(M),\mu_{\mathbf{x}_{0}}),\theta_{\mathcal{T}})$ is a realization of the colimit of the diagram $(L^{p}\left(H_{\mathcal{T}}(M),\nu_{\mathcal{T}}\right),\varrho_{\mathcal{T}\mathcal{T}'})$, where the morphisms $\theta_{\mathcal{T}}$, $\mathcal{T}\in\mathcal{P}$, are defined by 
\begin{equation*}
\theta_{\mathcal{T}}: L^{p}\left(H_{\mathcal{T}}(M),\nu_{\mathcal{T}}\right)\longrightarrow L^{p}(\mathcal{C}_{\mathbf{x}_{0}}(M),\mu_{\mathbf{x}_{0}}),\quad \theta_{\mathcal{T}}:=\Lambda^{-1}\circ \varphi_{\mathcal{T}}\circ \Lambda_{\mathcal{T}}.
\end{equation*}
In other words, the morphisms $\theta_{\mathcal{T}}$ are the unique morphisms making commutative the diagram of Figure 4. This identification establishes an analogue of Theorem \ref{th4.7} for the geometrical framework. The definition of the involved morphisms $\varrho_{\mathcal{T}\mathcal{T}'}, \theta_{\mathcal{T}}$ are given through the commutativity of the diagrams of Figure 3 and 4, respectively, in order to make things natural, in the categorical meaning of the word. This will become clear in the proof of the main Theorem of this subsection, Theorem \ref{T}.
\begin{center}
\begin{figure}[h!]
	\[
	\xymatrix{ 
		L^{p}\left(H_{\mathcal{T}}(M),\nu_{\mathcal{T}}\right) \ar^{\Lambda_{\mathcal{T}}}[d]  \ar@{-->}^{\theta_{\mathcal{T}}}[r]  &  L^{p}(\mathcal{C}_{\mathbf{x}_{0}}(M),\mu_{\mathbf{x}_{0}}) \ar^{\Lambda}[d] \\
		L^{p}(\mathbb{R}^{m\times\mathcal{T}},\lambda_{0}^{\mathcal{T}}) \ar^{\varphi_{\mathcal{T}}}[r] & L^{p}(\mathcal{C}_{0}(\mathbb{R}^{m}),\lambda_{0})
	}
	\]
	\label{F4}
	\caption{Definition of the morphisms $\theta_{\mathcal{T}}$}
\end{figure}
\end{center}

\begin{theorem}\label{T}
	The cocone $(L^{p}(\mathcal{C}_{\mathbf{x}_{0}}(M),\mu_{\mathbf{x}_{0}}),\theta_{\mathcal{T}})$ defines a realization of the colimit of $(L^{p}\left(H_{\mathcal{T}}(M),\nu_{\mathcal{T}}\right),\varrho_{\mathcal{T}\mathcal{T}'})$.
\end{theorem}

\begin{proof}
	By \eqref{EQG}, we have a family of isometric isomorphisms indexed by $\mathcal{T}$,
	\begin{equation*}
	\Lambda_{\mathcal{T}}:L^{p}\left(H_{\mathcal{T}}(M),\nu_{\mathcal{T}}\right)\to L^{p}(\mathbb{R}^{m\times \mathcal{T}},\mu^{\mathcal{T}}_{\mathbf{x}_{0}}).
	\end{equation*}
	By the commutativity of the diagram of Figure 3, this family of isomorphisms $(\Lambda_{\mathcal{T}})_{\mathcal{T}}$ establishes a natural isomorphism between the diagrams $(L^{p}\left(H_{\mathcal{T}}(M),\nu_{\mathcal{T}}\right),\varrho_{\mathcal{T}\mathcal{T}'})$ and $(L^{p}(\mathbb{R}^{m\times \mathcal{T}},\lambda_{0}^{\mathcal{T}}),\eta_{\mathcal{T}\mathcal{T}'})$. Hence the colimits of this two diagrams are isomorphic via a naturally defined isomorphism, see for instance the reference \cite[Cor. 3.6.3]{ER} for a proof of this fact. As the colimit of $(L^{p}(\mathbb{R}^{m\times \mathcal{T}},\lambda_{0}^{\mathcal{T}}),\eta_{\mathcal{T}\mathcal{T}'})$ is $(L^{p}(\mathcal{C}_{0}(\mathbb{R}^{m}),\lambda_{0}),\varphi_{\mathcal{T}})$, to prove the Theorem it is enough to prove that the cocone $(L^{p}(\mathcal{C}_{\mathbf{x}_{0}}(M),\mu_{\mathbf{x}_{0}}),\theta_{\mathcal{T}})$ is isomorphic to $(L^{p}(\mathcal{C}_{0}(\mathbb{R}^{m}),\lambda_{0}),\varphi_{\mathcal{T}})$ via a naturally defined isomorphism. The commutativity of the diagram of Figure 4, establishes that $\Lambda$ is our required isomorphism. This concludes the proof.
\end{proof}

\subsection{Derivation of the Limit Formula}

In this final subsection, we prove an analogue of Theorem \ref{C4.4} for the geometric measure $\nu_{\mathcal{T}}$. Let us denote by $\underrightarrow{\lim} \ L^{p}(\mathbb{R}^{m\times \mathcal{T}},\lambda_{0}^{\mathcal{T}})$ and $\underrightarrow{\lim} \ L^{p}\left(H_{\mathcal{T}}(M),\nu_{\mathcal{T}}\right)$, the analogues of the space defined by \eqref{E20}, under the natural changes to adapt it to $(\mathbb{R}^{m\times \mathcal{T}},\lambda_{0}^{\mathcal{T}})$ and $\left(H_{\mathcal{T}}(M),\nu_{\mathcal{T}}\right)$, respectively. We omit the explicit description of this spaces for notational convenience. Rephrasing the arguments of section 4, it can be proved that the operator
\begin{equation}\label{F7}
\mathfrak{I}_{p}:\underrightarrow{\lim} \ L^{p}(\mathbb{R}^{m\times \mathcal{T}},\lambda_{0}^{\mathcal{T}})\longrightarrow L^{p}(\mathcal{C}_{0}(\mathbb{R}^{m}),\lambda_{0}), \quad (f_{\mathcal{T}})_{\mathcal{T}\in\mathcal{P}}\mapsto \lim_{\mathcal{T}}\varphi_{\mathcal{T}}(f_{\mathcal{T}})
\end{equation}
is an isometric isomorphism. Thanks to this isomorphism and the identifications provided by \eqref{EQG} and \eqref{P7.3}, we prove the final result of this article, Theorem \ref{th7.2}. As a direct consequence of Theorem \ref{T} and the following Theorem \ref{th7.2}, we obtain that the cocone $$(\underrightarrow{\lim} \ L^{p}\left(H_{\mathcal{T}}(M),\nu_{\mathcal{T}}\right),\psi_{\mathcal{R}}),$$ where the morphisms $\psi_{\mathcal{R}}$ are the corresponding analogues of \eqref{An} for $(H_{\mathcal{T}}(M),\nu_{\mathcal{T}})$, defines a realization of the colimit of $(L^{p}\left(H_{\mathcal{T}}(M),\nu_{\mathcal{T}}\right),\varrho_{\mathcal{T}\mathcal{T}'})$.

\begin{theorem}\label{th7.2}
	Let $1\leq p<\infty$, then the following spaces are isometrically isomorphic
    \begin{equation}\label{ef}
\underrightarrow{\lim} \ L^{p}\left(H_{\mathcal{T}}(M),\nu_{\mathcal{T}}\right)\simeq L^{p}(\mathcal{C}_{\mathbf{x}_{0}}(M),\mu_{\mathbf{x}_{0}}).
    \end{equation}
	\noindent In consequence, for every $F\in L^{1}(\mathcal{C}_{\mathbf{x}_{0}}(M),\mu_{\mathbf{x}_{0}})$, there exists an element $(f_{\mathcal{T}})_{\mathcal{T}\in\mathcal{P}}\in\bigtimes_{\mathcal{T}\in\mathcal{P}}L^{1}\left(H_{\mathcal{T}}(M),\nu_{\mathcal{T}}\right)$ such that 
	\begin{equation}\label{TF2}
	\int_{\mathcal{C}_{\mathbf{x}_{0}}(M)}F \ d\mu_{\mathbf{x}_{0}}=\lim_{\mathcal{T}}\int_{H_{\mathcal{T}}(M)}f_{\mathcal{T}}\ d\nu^{\mathcal{T}}.
	\end{equation}
\end{theorem}

\begin{proof}
	By \eqref{EQG}, we have a family of isometric isomorphisms indexed by $\mathcal{T}$,
	\begin{equation*}
	\Lambda_{\mathcal{T}}:L^{p}\left(H_{\mathcal{T}}(M),\nu_{\mathcal{T}}\right)\to L^{p}(\mathbb{R}^{m\times \mathcal{T}},\mu^{\mathcal{T}}_{\mathbf{x}_{0}}).
	\end{equation*}
	It is easily seen using the definition of the morphisms $\varrho_{\mathcal{T}\mathcal{T}'}$, Figure 3, that the operator 
	\begin{equation*}
	\Sigma_{p}:\underrightarrow{\lim} \ L^{p}\left(H_{\mathcal{T}}(M),\nu_{\mathcal{T}}\right)\longrightarrow \underrightarrow{\lim} \ L^{p}(\mathbb{R}^{m\times \mathcal{T}},\lambda_{0}^{\mathcal{T}}), \quad \Sigma_{p}(f_{\mathcal{T}})_{\mathcal{T}\in\mathcal{P}}:=(\Lambda_{\mathcal{T}}(f_{\mathcal{T}}))_{\mathcal{T}\in\mathcal{P}}
	\end{equation*}
	defines an isometric isomorphism. Finally, the composition operator
		\[
	\xymatrix{ 
		\underrightarrow{\lim} \ L^{p}\left(H_{\mathcal{T}}(M),\nu_{\mathcal{T}}\right) \ar^{\Sigma_{p}}[d]  \ar@{-->}[r]  &  L^{p}(\mathcal{C}_{\mathbf{x}_{0}}(M),\mu_{\mathbf{x}_{0}}) \\
		\underrightarrow{\lim} \ L^{p}(\mathbb{R}^{m\times \mathcal{T}},\lambda_{0}^{\mathcal{T}}) \ar^{\mathfrak{I}_{p}}[r] & L^{p}(\mathcal{C}_{0}(\mathbb{R}^{m}),\lambda_{0}) \ar^{\Lambda^{-1}}[u]
	}
	\]
	where the morphisms $\mathfrak{I}_{p}$ and $\Lambda$ are given by \eqref{F7} and \eqref{P7.3}, respectively, defines an isometric isomorphism. This concludes the proof of \eqref{ef}. Formula \eqref{TF2} follows from the isometric property of the induced operator.
\end{proof}

\section{Acknowledgements}

The author is very grateful to the animous referee for his/her extremely useful advises that without any doubt has improved the presentation and correctness of the article and has improved deeply the results of it. Moreover, the author is also very fortunate to have been advised by Professors Christian Bär and Matthias Ludewig, two experts in this field.

\end{document}